\documentclass[11pt]{amsart}
\usepackage{bbm}
\usepackage{amssymb, amsthm, amsmath, amscd}
\usepackage[usenames,dvipsnames]{color}

\theoremstyle{definition}
\newtheorem{theorem}{Theorem}[section]
\newtheorem{proposition}{Proposition}[section]
\newtheorem{corollary}{Corollary}[section]
\newtheorem{lemma}{Lemma}[section]
\newtheorem{notation}{Notation}[section]
\newtheorem{remark}{Remark}[section]
\newtheorem{definition}{Definition}[section]

\usepackage{amsfonts}
\usepackage{mathrsfs}
\usepackage{enumerate}
\usepackage{textcomp}
\usepackage[all]{xy}

\begin{document}

\title{A note on the weak tracial Rokhlin property for finite group actions on simple unital $\rm{C}^*$-algebras}

\author{Xiaochun Fang}
\address{School of Mathematical Sciences, Tongji University, Shanghai 200092, China}
\email{xfang@tongji.edu.cn}

\author{Zhongli Wang}
\address{School of Mathematical Sciences, Tongji University, Shanghai 200092, China}
\email{1810411@tongji.edu.cn}

\date{Month, Day, Year}
\keywords{C*-algebra; Tracial Rokhlin property; Weak tracial Rokhlin property}
\date{\today}

\maketitle
\begin{abstract}
In this paper, we show that one of the conditions in the definition of weak tracial Rokhlin property for finite group actions on simple unital $\rm{C}^*$-algebras can be replaced by a seemingly weaker condition, or a seemingly stronger condition. As a corollary, this condition is redundant whenever the $\rm{C}^*$-algebra is not purely infinite. We also give a sufficient condition for the weak tracial Rokhlin property for finite group actions on simple unital $\rm{C}^*$-algebras to imply the tracial Rokhlin property.
\end{abstract}








\section{INTRODUCTION}
\label{sec1}

The Rokhlin property was originally used in ergodic theory for von Neumann algebras. The use of the Rokhlin property for finite group actions on $\rm{C}^*$-algebras dates back to \cite{MR647069,MR702960}. The systematic study of finite group actions on $\rm{C}^*$-algebras with the Rokhlin property was initiated by Izumi in \cite{MR2053753,MR2047851}. The Rokhlin property can be viewed as a regularity condition for group action, which can be used to show that various structural properties pass from a $\rm{C}^*$-algebra to its crossed product.

However, actions of finite groups with the Rokhlin property are very rare. Phillips introduced the tracial Rokhlin property for finite group actions on simple unital $\rm{C}^*$-algebras in \cite{MR2808327}. The tracial Rokhlin property is much more common, and is also useful in studying the structure of crossed products. For example, Phillips proved that the crossed product of a simple unital $\rm{C}^*$-algebra with tracial rank zero by a finite group action with the tracial Rokhlin property again has tracial rank zero(see \cite[Theorem 2.6]{MR2808327}).

Nevertheless, the tracial Rokhlin property still has restrictions. It requires the existence of projections. For example, the Jiang-Su algebra $\mathcal{Z}$ does not admit any action with the tracial Rokhlin property. Several weak versions of the tracial Rokhlin property in which one uses orthogonal positive elements instead of orthogonal projections were studied for finite group actions on simple unital $\rm{C}^*$-algebras, see \cite{MR2712082,MR2954514,MR3063095,MR3153204,MR4177290,MR4336489}. As an example, the permutation action of $S_n$ on the Jiang-Su algebra $\mathcal{Z}\cong \mathcal{Z}^{\otimes n}$ has the weak tracial Rokhlin property but it does not have the tracial Rokhlin property(see \cite[Example 5.10]{MR3063095}). The main use of the weak tracial Rokhlin property has been showing that the crossed product of a simple separable unital nuclear $\mathcal{Z}$-stable $\rm{C}^*$-algebra by a finite group action with the weak tracial Rokhlin property again is $\mathcal{Z}$-stable(see \cite[Corollary 5.7]{MR3063095}). Let us recall the definition of the weak tracial Rokhlin property.

\begin{definition}\label{def1}(\cite[Definition 3.2]{MR4336489})
	Let $A$ be a simple unital $\rm{C}^*$-algebra, let $G$ be a finite group, and let $\alpha:G\rightarrow \text{Aut}(A)$ be an action of $G$ on $A$.
	\begin{enumerate}[(1)]
		\item $\alpha$ is said to have the weak tracial Rokhlin property, if for any finite subset $F\subset A$, any $\varepsilon>0$, and any $x\in A_+$ with $\|x\|=1$, there exist orthogonal positive contractions $f_g\in A$ for $g\in G$, with $f=\sum_{g\in G}f_g$, such that:
		\begin{enumerate}[(a)]
			\item $\|f_ga-af_g\|<\varepsilon$ for all $g\in G$ and all $a\in F$;
			\item $\|\alpha_g(f_h)-f_{gh}\|<\varepsilon$ for all $g,h\in G$;
			\item $1-f\precsim x$;
			\item $\|fxf\|>1-\varepsilon$.
		\end{enumerate}
		\item $\alpha$ is said to have the tracial Rokhlin property if the orthogonal positive contractions in $(1)$ can be chosen to be orthogonal projections.
	\end{enumerate}
\end{definition}

If $A$ is purely infinite, then Condition (c) is automatic. If $A$ is finite, then Condition (d) is redundant, see \cite[Lemma 1.16]{MR2808327} for the tracial Rokhlin property and \cite[Lemma 2.9]{1408.5546} for the weak tracial Rokhlin property. However, without Condition (d), the trivial action on a purely infinite simple unital $\rm{C}^*$-algebra would have the weak tracial Rokhlin property. It is not clear that Condition (d) is really the right extra condition to impose. 

In the first part of this paper ,using the method in \cite[Proposition 9.5]{MR4215379}, we show that Condition (d) in Definition \ref{def1} can be replaced by a seemingly weaker condition, or a seemingly stronger condition. As a corollary, Condition (d) is redundant whenever $A$ is not purely infinite.

\begin{theorem}\label{the1}
	Let $A$ be a simple unital $\rm{C}^*$-algebra, let $G$ be a finite group, and let $\alpha:G\rightarrow \text{Aut}(A)$ be an action of $G$ on $A$. Then the following statements are equivalent:
	\begin{enumerate}[(1)]
		\item $\alpha$ has the weak tracial Rokhlin property (respectively, tracial Rokhlin property);
		\item For any finite subset $F\subset A$, any $\varepsilon>0$, and any $x\in A_+\setminus\{0\}$, there exist orthogonal positive contractions (respectively, orthogonal projections) $f_g\in A$ for $g\in G$ such that, with $f=\sum_{g\in G}f_g$, the following hold:
		\begin{enumerate}[(a)]
			\item $\|f_ga-af_g\|<\varepsilon$ for all $g\in G$ and all $a\in F$;
			\item $\|\alpha_g(f_h)-f_{gh}\|<\varepsilon$ for all $g,h\in G$;
			\item $1-f\precsim x$;
			\item $\|f\|>1-\varepsilon$.
		\end{enumerate}
		\item For any finite subset $F\subset A$, any $\varepsilon>0$, and any $x\in A_+\setminus\{0\}$, there exist orthogonal positive contractions (respectively, orthogonal projections) $f_g\in A$ for $g\in G$ such that, with $f=\sum_{g\in G}f_g$, the following hold:
		\begin{enumerate}[(a)]
			\item $\|f_ga-af_g\|<\varepsilon$ for all $g\in G$ and all $a\in F$;
			\item $\|\alpha_g(f_h)-f_{gh}\|<\varepsilon$ for all $g,h\in G$;
			\item $1-f\precsim x$;
			\item $\|faf\|>\|a\|-\varepsilon$ for all $a\in F$.
		\end{enumerate}
	\end{enumerate}
\end{theorem}

As one could expect, for finite group actions on $\rm{C}^*$-algebras which have sufficiently many projections, the weak tracial Rokhlin property is equivalent to the tracial Rokhlin property. This was proved in \cite[Theorem 1.9]{MR3347172} for $\rm{C}^*$-algebras with tracial rank zero, and in \cite[Theorem 3.11]{MR4177290} for unital Kirchberg algebras. In the second part of this paper, we give a class of simple unital $\rm{C}^*$-algebras which contains both the class of $\rm{C}^*$-algebras with tracial rank zero and the class of unital Kirchberg algebras. For finite group actions on $\rm{C}^*$-algebras in this class, the weak tracial Rokhlin property implies the tracial Rokhlin property.

\begin{theorem}\label{the2}
	Let $A$ be an infinite-dimensional simple unital $\rm{C}^*$-algebra, let $G$ be a finite group, let $\alpha:G\rightarrow \text{Aut}(A)$ be an action of $G$ on $A$, and let $\omega$ be a free ultrafilter on $\mathbb{N}$. Suppose that for any finite subset $F\subset A$, any $\varepsilon>0$, and any $x\in A_+\setminus\{0\}$, there exist a separable $C^*$-subalgebra $B\subset A$ and a positive element $c\in A\cap B'$ with $\|c\|=1$, such that:
	\begin{enumerate}[(1)]
		\item $\|ca-ac\|<\varepsilon$ for all $a\in F$;
		\item $ca\in_\varepsilon B$ for all $a\in F$;
		\item $1-c\precsim x$;
		\item $A_\omega\cap B'$ has real rank zero.
	\end{enumerate}
	Then $\alpha$ has the tracial Rokhlin property if and only if $\alpha$ has the weak tracial Rokhlin property.
\end{theorem}

This paper is organized as follows. In Section \ref{sec2}, we present some definitions and known results about Cuntz comparison, order zero map, approximation and ultrapower. Then we give our main results in Section \ref{sec3}.

\section{PRELIMINARIES AND DEFINITIONS}
\label{sec2}
\begin{notation} We use the following notation in this paper.
	\begin{enumerate}[(1)]
		\item For a $\rm{C}^*$-algebra $A$, let $A_+$ denote the set of all positive elements in $A$, $A^1$ denote the closed unit ball of $A$, and $A_+^1$ denote the set of all positive contractions in $A$.
		\item Let $A$ be a $\rm{C}^*$-algebra, $F\subset A$ be a subset, and $a\in A$. We write $a\in_\varepsilon F$ if there is $b\in F$ such that $\|a-b\|<\varepsilon$.
		\item Let $A$ be a $\rm{C}^*$-algebra. For $a,b\in A_+$, we say that $a$ is Cuntz subequivalent to $b$ in $A$, written $a\precsim_A b$, if there is a sequence $(v_n)_{n=1}^\infty$ in $A$ such that $\lim_{n\rightarrow \infty}\|a-v_nbv_n^*\|=0$. We say that $a$ is Cuntz equivalent to $b$, written $a\sim_A b$ if both $a\precsim_A b$ and $b\precsim_A a$. The relation $\sim_A$ is an equivalence relation. When there is no confusion about the algebra $A$, we suppress it in the notation.
		\item Let $A$ be a $\rm{C}^*$-algebra, $a\in A_+$, and $\varepsilon>0$. Define a continuous function $f_\varepsilon:[0,\infty)\rightarrow [0,\infty)$ by
		\begin{equation*}
			f_\varepsilon(t):= \begin{cases}
				0 \quad & 0\leq t\leq \varepsilon\\
				t-\varepsilon \quad & \varepsilon < t.
			\end{cases}
		\end{equation*}
		We use $(a-\varepsilon)_+$ to denote $f_\varepsilon(a)$. By the functional calculus, it follows in a straightforward manner that $((a-\varepsilon_1)_+-\varepsilon_2)_+=(a-(\varepsilon_1+\varepsilon_2))_+$ for all $\varepsilon_1,\varepsilon_2>0$(see \cite[Lemma 2.5(i)]{MR1759891}).
	\end{enumerate}	
\end{notation}

The following facts abous Cuntz subequivalence are well known. Part (1) is \cite[Lemma 2.2]{MR1906257}. Part (2) is \cite[Lemma 2.7]{MR4336489}. Part (3) is \cite[Lemma 1.7]{1408.5546}.

\begin{lemma}\label{lem1} Let $A$ be a $\rm{C}^*$-algebra.
	\begin{enumerate}[(1)]
		\item Let $a,b\in A_+$, and let $\varepsilon>0$. If $\|a-b\|<\varepsilon$, then there is a contraction $d\in A$ such that $(a-\varepsilon)_+=dbd^*$. In particular, $(a-\varepsilon)_+\precsim b$.
		\item Let $a,b\in A_+$, and let $\delta>0$. If $a\precsim (b-\delta)_+$, then there exists a sequence $(v_n)_{n=1}^\infty$ in $A$ such that $\|a-v_nbv_n^*\|\rightarrow 0$ and $\|v_n\|\leq \|a\|^{1/2}\delta^{-1/2}$ for every $n\in \mathbb{N}$.
		\item Let $a,b\in A$ satisfy $0\leq a\leq b$. Let $\varepsilon>0$. Then $(a-\varepsilon)_+\precsim (b-\varepsilon)_+$.
	\end{enumerate}
	
\end{lemma}

The following definition,  withou Condition (d) in Definition \ref{def1} but requiring $\|f_g\|=1$, was introduced by Hirshberg and Orovitz in \cite{MR3063095} under the name generalized tracial Rokhlin property.

\begin{definition}\label{def2}(\cite[Definition 5.2]{MR3063095})
	Let $A$ be a simple unital $\rm{C}^*$-algebra, let $G$ be a finite group, and let $\alpha:G\rightarrow \text{Aut}(A)$ be an action of $G$ on $A$. Then $\alpha$ is said to have the generalized tracial Rokhlin property, if for any finite subset $F\subset A$, any $\varepsilon>0$, and any $x\in A_+\setminus\{0\}$, there exist normalized orthogonal positive elements $f_g\in A$ for $g\in G$, with $f=\sum_{g\in G}f_g$, such that:
	\begin{enumerate}[(1)]
		\item $\|f_ga-af_g\|<\varepsilon$ for all $g\in G$ and all $a\in F$;
		\item $\|\alpha_g(f_h)-f_{gh}\|<\varepsilon$ for all $g,h\in G$;
		\item $1-f\precsim x$.
	\end{enumerate}
\end{definition}

\begin{definition}(\cite[Definition 2.3]{MR2545617})
	Let $A,B$ be $\rm{C}^*$-algebras, and let $\varphi:A\rightarrow B$ be a completely positive map. Then $\varphi$ is said to have order zero, if $\varphi(a)\varphi(b)=0$ for all $a,b\in A_+$ with $ab=0$.
\end{definition}

The following lemma is part of \cite[Proposition 5.3]{MR4464578}.

\begin{lemma}\label{lem2}(\cite[Proposition 5.3]{MR4464578})
	Let $A$ and $B$ be $\rm{C}^*$-algebras, and let $\varphi:A\rightarrow B$ be a completely positive order zero map. If $A$ is simple, then $\|\varphi(a)\|=\|\varphi\|\cdot \|a\|$ for all $a\in A$.
\end{lemma}

\begin{lemma}\label{lem4}(\cite[Lemma 2.6]{1408.5546})
	Let $A$ be a simple $\rm{C}^*$-algebra, and let $B\subset A$ be a non-zero hereditary subalgebra. Let $n\in \mathbb{N}$, and let $a_1,a_2,\cdots,a_n\in A_+\setminus\{0\}$. Then there exists $b\in B_+\setminus\{0\}$ such that $b\precsim a_j$ for $j=1,2,\cdots,n$.
\end{lemma}

We need the following approximation lemma. Part (1) is \cite[Lemma 1.5]{MR4078704}. Part (2) is \cite[Lemma 1.6]{MR4078704}.

\begin{lemma}\label{lem3} Suppose $f:[0,1]\rightarrow \mathbb{C}$ is continuous.
	\begin{enumerate}[(1)]
		\item For every $\varepsilon>0$ there exists $\delta>0$ such that whenever $A$ is a $\rm{C}^*$-algebra and $a,b\in A$ satisfy
		\[
		\|a\|\leq 1, 0\leq b\leq 1, \text{ and } \|[b,a]\|<\delta,
		\]
		then $\|[f(b),a]\|<\varepsilon$.
		\item If $f(0)=0$, then for every $\varepsilon>0$ there exists $\delta>0$ such that whenever $A$ is a $\rm{C}^*$-algebra, $B\subset A$ is a subalgebra, and $a\in A,b\in B$ satisfy
		\[
		\|a\|\leq 1, 0\leq b\leq 1, \text{ and } dist(ba,B)<\delta,
		\]
		then $dist(f(b)a,B)<\varepsilon$.
	\end{enumerate}
\end{lemma}

We recall the notion of ultrapower of $\rm{C}^*$-algebras.

\begin{notation} Let $A$ be a $\rm{C}^*$-algebra, and let $l^\infty(A)$ denote the $\rm{C}^*$-algebra of all bounded functions from $\mathbb{N}$ into $A$ with entry-wise defined algebraic operations. Let $\omega$ be a free ultrafilter on $\mathbb{N}$. The ultrapower of $A$ is defined to be
	\[
	A_\omega:=l^\infty(A)/\{(a_n)_{n=1}^\infty\in l^\infty(A):\lim_{n\rightarrow \omega}\|a_n\|=0\}.
	\]
	We identity $A$ with the $\rm{C}^*$-subalgebra of $A_\omega$ consisting of equivalence classes of constant sequences. For a subset $B\subset A$, we denote by $A_\omega\cap B'$ the relative commutant of $B$ in $A_\omega$. For an action $\alpha:G\rightarrow \text{Aut}(A)$, we denote by $\alpha_\omega$ the induced action of $G$ on $A_\omega$.
\end{notation}

\section{THE MAIN RESULTS}\label{sec3}

First, we give the proof of Theorem \ref{the1}.

\begin{proof}[Proof of Theorem \ref{the1}]
	We only prove the case of the weak tracial Rokhlin property. The proof for the tracial Rokhlin property is the same.\\
	(2) $\Rightarrow$ (3): Fix $\varepsilon>0$, a finite subset $F\subset A$, and a non-zero positive element $x\in A$. Without loss of generality, we may assume $1_A\in F$. Since the property \textquotedblleft$A$ is simple\textquotedblright \  is an separably inheritable property(see \cite[II.8.5.6]{MR2188261}), there is a unital separable simple $\rm{C}^*$-subalgebra $B\subset A$ which contains $F$. Fix a dense sequence $(b_n)_{n=1}^\infty$ in $B$. For each $n\in \mathbb{N}$, applying statement (2) with $F_n=\{b_1,b_2,\cdots,b_n\}$ in place of $F$, $1/n$ in place of $\varepsilon$, and $x$ as given, there exist orthogonal positive contractions $f_g^{(n)}\in A$ for $g\in G$ such that, with $f^{(n)}=\sum_{g\in G}f_g^{(n)}$, the following hold:
	\begin{enumerate}[(a)]
		\item $\|f_g^{(n)}b-bf_g^{(n)}\|<\frac{1}{n}$ for all $g\in G$ and all $b\in F_n$;
		\item $\|\alpha_g(f_h^{(n)})-f_{gh}^{(n)}\|<\frac{1}{n}$ for all $g,h\in G$;
		\item $1-f^{(n)}\precsim x$;
		\item $\|f^{(n)}\|>1-\frac{1}{n}$.
	\end{enumerate}
	Thus we get sequences of positive contractions $(f_g^{(n)})_{n=1}^\infty$ for $g\in G$ such that
	\[
	\lim_{n\rightarrow \infty}\|f_g^{(n)}b-bf_g^{(n)}\|=0 \text{ for all } g\in G \text{ and all }b\in B, \text{ and}
	\]
	\[
	\lim_{n\rightarrow \infty}\|\alpha_g(f_h^{(n)})-f_{gh}^{(n)}\|=0\text{ for all } g,h\in G.
	\]
	Since $G$ is a finite group, we have
	\[
	\lim_{n\rightarrow \infty}\|f^{(n)}b-bf^{(n)}\|=0 \text{ for all } b\in B.
	\]
	For each $n\in \mathbb{N}$, define a completely positive linear map $\varphi_n:B\rightarrow A$ by
	\[
	\varphi_n(b):=f^{(n)}bf^{(n)} \text{ for all } b\in B.
	\]
	Choosing a free ultrafilter $\omega$ on $\mathbb{N}$, we may consider the ultrapower and the completely positive map $\varPhi:B\rightarrow A_\omega$, which is induced by $(\varphi_n)_{n=1}^\infty$. We claim that $\varPhi$ has order zero. In fact, for any $b,c\in B_+^1$ with $bc=0$, we have
	\begin{align*}
		\|\varPhi(b)\varPhi(c)\| & =\mathop{\lim}\limits_{n\rightarrow \omega}\|f^{(n)}bf^{(n)}f^{(n)}cf^{(n)}\|\\
		& \leq \mathop{\lim}\limits_{n\rightarrow \omega}\|f^{(n)}bc(f^{(n)})^3\|+\mathop{\lim}\limits_{n\rightarrow \omega}\|(f^{(n)})^2c-c(f^{(n)})^2\|=0.
	\end{align*}
	This proves the claim. Since $B$ is simple and $\|\varPhi\|=\|\varPhi(1)\|=\mathop{\lim}\limits_{n\rightarrow \omega}\|(f^{(n)})^2\|=1$, it follows from Lemma \ref{lem2} that $\varPhi$ is an isometry. Since $F$ is a finite set, we can find an increasing subsequence $(n_k)_{k=1}^{\infty}$ such that
	\[
	\|f^{(n_k)}af^{(n_k)}\|>\|a\|-\varepsilon \text{ for all } a\in F \text{ and } k\in \mathbb{N}.
	\]
	Note that we have
	\[
	\lim_{k\rightarrow \infty}\|f_g^{(n_k)}a-af_g^{(n_k)}\|=0 \text{ for all } g\in G \text{ and all }a\in F,
	\]
	\[\lim_{k\rightarrow \infty}\|\alpha_g(f_h^{(n_k)})-f_{gh}^{(n_k)}\|=0 \text{ for all } g,h\in G, \text{ and}
	\]
	\[
	1-f^{(n_k)}\precsim x \text{ for all } k\in \mathbb{N}.
	\]
	Choosing a sufficiently large $k$ and set $f_g:=f_g^{(n_k)}$ for $g\in G$, the conclusion of (3) holds.\\
	(3) $\Rightarrow$ (1) is immediate.\\
	(1) $\Rightarrow$ (2): Fix $\varepsilon>0$, a finite subset $F\subset A$, and a non-zero positive element $x\in A$. Without loss of generality, we can assume that $\|x\|=1$. Let $f_g$ and $f$ are as in Definition \ref{def1}, then Conditions (a)\textendash(c) in statement (2) are satisfied. Note that $(f_g)_{g\in G}$ are pairwise orthogonal, thus $\|f\|=\text{max}\{\|f_g\|:g\in G\}\leq 1$. Using Condition (d) in Definition \ref{def1} to get the last inequality, we have
	\[
	\|f\|\geq\|f\|^2=\|f\|\|x\|\|f\|\geq\|fxf\|>1-\varepsilon.
	\]
\end{proof}

\begin{corollary}
	Let $A$ be a simple unital $\rm{C}^*$-algebra, let $G$ be a finite group, and let $\alpha:G\rightarrow \text{Aut}(A)$ be an action of $G$ on $A$. Then the following statements are equivalent:
	\begin{enumerate}[(1)]
		\item $\alpha$ has the generalized tracial Rokhlin property;
		\item $\alpha$ has the weak tracial Rokhlin property.
	\end{enumerate}
\end{corollary}
\begin{proof}
	(1) $\Rightarrow$ (2) is contained in Theorem \ref{the1}.\\
	(2) $\Rightarrow$ (1): Fix $\varepsilon>0$, a finite subset $F\subset A$, and a non-zero positive element $x\in A$. Without loss of generality, we can assume that $\|x\|=1$. Choose $\delta$ with $0<\delta<\frac{1}{2}$ such that
	\[
	\frac{2\delta}{1-2\delta}<\varepsilon.
	\]
	Since $\alpha$ has the weak tracial Rokhlin property, there exist orthogonal positive contractions $e_g\in A$ for $g\in G$ such that, with $e=\sum_{g\in G}e_g$, the following hold:
	\begin{enumerate}[(a)]
		\item $\|e_ga-ae_g\|<\delta$ for all all $g\in G$ and $a\in F$;
		\item $\|\alpha_g(e_h)-e_{gh}\|<\delta$ for all $g,h\in G$;
		\item $1-e\precsim x$;
		\item $\|exe\|>1-\delta$.
	\end{enumerate}
	Note that $\|e\|=\text{max}\{\|e_g\|:g\in G\}$. Hence, by (d), there exists $g_0\in G$ such that
	\[
	\|e_{g_0}\|=\|e\|\geq \|exe\|>1-\delta.
	\]
	Using (b) we have
	\[
	\|e_g\|>\|\alpha_{gg_0^{-1}}(e_{g_0})\|-\delta=\|e_{g_0}\|-\delta>1-2\delta>0 \text{ for all }g\in G.
	\]
	Set $f_g=e_g/\|e_g\|$ for each $g\in G$. Then $(f_g)_{g\in G}$ is a family of normalized orthogonal positive contractions in $A$. Using (a) and (b) we have
	\[
	\|f_ga-af_g\|=\frac{\|e_ga-ae_g\|}{\|e_g\|}<\frac{\delta}{1-2\delta}<\varepsilon,
	\]
	\[
	\|\alpha_g(f_h)-f_{gh}\|\leq\|\frac{\alpha_g(e_h)-e_{gh}}{\|e_h\|}\|+\|\frac{e_{gh}}{\|e_h\|}-\frac{e_{gh}}{\|e_{gh}\|}\|<\frac{2\delta}{1-2\delta}<\varepsilon
	\]
	for all $a\in F$ and all $g,h\in G$. Moreover, put $f=\sum_{g\in G}f_g$, using (c) we have 
	\[
	1-f\leq 1-e\precsim x.
	\]
\end{proof}

Note that every non-zero projection is norm 1. Thus we have the following corollary. This corollary says that Condition (d) in the definition of tracial Rokhlin property is redundant if we require one of the projections to be non-zero.

\begin{corollary}\label{cor1}
	Let $A$ be a simple unital $\rm{C}^*$-algebra, let $G$ be a finite group, and let $\alpha:G\rightarrow \text{Aut}(A)$ be an action of $G$ on $A$. If for any finite subset $F\subset A$, any $\varepsilon>0$, and any $x\in A_+$, there are non-zero mutually orthogonal projections $e_g\in A$ for $g\in G$, with $e=\sum_{g\in G}e_g$, the following hold:
	\begin{enumerate}[(1)]
		\item $\|e_ga-ae_g\|<\varepsilon$ for all $g\in G$ and all $a\in F$;
		\item $\|\alpha_g(e_h)-e_{gh}\|<\varepsilon$ for all $g,h\in G$;
		\item $1-e\precsim x$.
	\end{enumerate}
	Then $\alpha$ has the tracial Rokhlin property.
\end{corollary}

Recall that a simple unital $\rm{C}^*$-algebra $A$ is called purely infinite if $A$ is infinite dimensional and for any non-zero element $a\in A$, there are $x,y\in A$ such that $xay=1$. It is well known that a simple infinite dimensional unital $\rm{C}^*$-algebra $A$ is purely infinite if and only if $a\precsim b$ for any non-zero $a,b\in A_+$(see \cite{MR1124300}). Condition (d) in Definition \ref{def1} is needed to ensure that the trivial action on a purely infinite simple unital $\rm{C}^*$-algebra does not have the weak tracial Rokhlin porperty. As a consequence of Theorem \ref{the1}, Condition (d) in Definition \ref{def1} is redundant when $A$ is not purely infinite.

The following lemma is a simple version of \cite[Lemma 2.9]{1408.5546}.

\begin{lemma}
	Let $A$ be a simple infinite dimensional unital $\rm{C}^*$-algebra which is not purely infinite. Let $x\in A_+\setminus\{0\}$. Then for any $\varepsilon>0$ there is $y\in \overline{(xAx)}_+\setminus \{0\}$ such that whenever $f\in A_+$ satisfies $0\leq f\leq 1$ and $1-f\precsim y$, then $\|f\|>1-\varepsilon$.
\end{lemma}
\begin{proof}
	Since $A$ is not purely infinite, there exists $a\in A_+\setminus\{0\}$ such that $a$ is not Cuntz equivalent to $1_A$. Applying Lemma \ref{lem4} with $\overline{xAx}$ in place of $B$, $n=1$, and $a$ in place of $a_1$, we get $y\in \overline{(xAx)}_+\setminus \{0\}$ such that $y\precsim a$. Let $f\in A_+$ satisfy $0\leq f\leq 1_A$ and $1_A-f\precsim y$, we want to show that $\|f\|>1-\varepsilon$. Suppose that $\|f\|\leq 1-\varepsilon$. Since $\|1_A-(1_A-f)\|<1-\varepsilon/2$, by Lemma \ref{lem1}(1) and the choice of $y$, we have
	\[
	1_A\sim \frac{\varepsilon}{2}\cdot 1_A=(1_A-(1-\frac{\varepsilon}{2}))_+\precsim 1_A-f\precsim y\precsim a.
	\]
	Thus $a\sim 1_A$, a contradiction.
\end{proof}

\begin{corollary}
	Let $A$ be a simple infinite dimensional unital $\rm{C}^*$-algebra which is not purely infinite, let $G$ be a finite group, and let $\alpha:G\rightarrow \text{Aut}(A)$ be an action of $G$ on $A$. If for any finite subset $F\subset A$, any $\varepsilon>0$, and any $x\in A_+\setminus\{0\}$, there exist orthogonal positive contractions $f_g\in A$ for $g\in G$ such that, with $f=\sum_{g\in G}f_g$, the following hold:
	\begin{enumerate}[(1)]
		\item $\|f_ga-af_g\|<\varepsilon$ for all $g\in G$ and all $a\in F$;
		\item $\|\alpha_g(f_h)-f_{gh}\|<\varepsilon$ for all $g,h\in G$;
		\item $1-f\precsim x$.
	\end{enumerate}
	Then $\alpha$ has the weak tracial Rokhlin property.
\end{corollary}

\begin{proof}
	Let $F\subset A$ be a finite subset, let $\varepsilon>0$, and let $x\in A_+\setminus\{0\}$. Applying the above lemma, we get $y\in \overline{(xAx)}_+\setminus \{0\}$ such that whenever $f\in A_+$ satisfies $0\leq f\leq 1$ and $1-f\precsim y$, then $\|f\|>1-\varepsilon$. Applying the hypothesis with $F,\varepsilon$ as given, and with $y$ in place of $x$, we get orthogonal positive contractions $f_g\in A$ for $g\in G$. By Theorem \ref{the1}, we only need to prove that $\|f\|>1-\varepsilon$. This follows from the choice of $y$.
\end{proof}

The following lemma is well known.
\begin{lemma}\label{lem6}
	Let $A$ be a $C^*$-algebra, let $\omega$ be a free ultrafilter on $\mathbb{N}$, and let $\pi:l^\infty(A)\rightarrow A_\omega$ be the quotient map. Let $(q_n)_{n\in \mathbb{N}}$ be a family of orthogonal projections in $A_\omega$. Then there exists a family $(p_n)_{n\in \mathbb{N}}$ of orthogonal projections in $l^\infty(A)$ such that $\pi(p_n)=q_n$ for all $n\in \mathbb{N}$.
\end{lemma}
\begin{proof}
	By \cite[Lemma 10.1.12]{MR1420863}, there is a family $(f_n)_{n\in \mathbb{N}}$ of orthogonal positive contractions in $l^\infty(A)$ such that $\pi(f_n)=q_n$ for all $n\in \mathbb{N}$. For each $n\in \mathbb{N}$, write $f_n=(f_n^{(1)},f_n^{(2)},\cdots)\in l^\infty(A)$. Since $q_n$ is a projection, we can choose $X_n$ in $\omega$ such that $\|f_n^{(k)}-(f_n^{(k)})^2\|<1/4$ for all $k\in X_n$. By Lemma \cite[Lemma 2.5.5]{MR1884366}, for each $k\in X_n$ there exist a projection $p_n^{(k)}$ in the $C^*$-subalgebra generated by $f_n^{(k)}$ such that
	\[
	\|f_n^{(k)}-p_n^{(k)}\|<2\|f_n^{(k)}-(f_n^{(k)})^2\|.
	\]
	Set $p_n^{(k)}=0$ if $k$ does not belong to $X_n$, and put $p_n=(p_n^{(1)},p_n^{(2)},\cdots)$. Then $p_n$ is a projection in $l^\infty(A)$, $\lim_{n\rightarrow \omega}\|p_n^{(k)}-f_n^{(k)}\|=0$, and hence $\pi(p_n)=q_n$. It remains to show that $(p_n)_{n\in \mathbb{N}}$ are pairwise orthogonal, or equivalently, $(p_n^{(k)})_{n\in \mathbb{N}}$ are pairwise orthogonal for each fixed $k$. This follows from the fact that $(f_n^{(k)})_{n\in \mathbb{N}}$ are pairwise orthogonal, and $p_n^{(k)}$ is in the $C^*$-subalgebra generated by $f_n^{(k)}$.
\end{proof}

Using the method and technique in \cite[Proposition 3.10]{MR4216445}, we give an equivalent reformulation of the (weak) tracial Rokhlin property for finite group actions on simple unital $\rm{C}^*$-algebras using the ultrapower.

\begin{proposition}\label{pro1}
	Let $A$ be a simple unital $\rm{C}^*$-algebra, let $G$ be a finite group, let $\alpha:G\rightarrow \text{Aut}(A)$ be an action of $G$ on $A$, and let $\omega$ be a free ultrafilter on $\mathbb{N}$. Then the following statements are equivalent:
	\begin{enumerate}[(1)]
		\item $\alpha$ has the weak tracial Rokhlin property (respectively, tracial Rokhlin property);
		\item For any separable subset $S$ of $A_\omega$, and any $x\in (A_\omega)_+\setminus\{0\}$, there exist orthogonal positive contractions (respectively, orthogonal projections) $f_g\in A_\omega\cap S'$ for $g\in G$ such that, with $f=\sum_{g\in G}f_g$, the following hold:
		\begin{enumerate}[(a)]
			\item $(\alpha_\omega)_g(f_h)=f_{gh}$ for all $g,h\in G$;
			\item $1-f\precsim_{A_\omega} x$;
			\item $\|fsf\|=\|s\|$ for all $s\in S$.
		\end{enumerate}
		\item For any finite subset $F\subset A$, any $\varepsilon>0$, and any $x\in A_+\setminus\{0\}$, there exist orthogonal positive contractions (respectively, orthogonal projections) $f_g\in A_\omega$ for $g\in G$ such that, with $f=\sum_{g\in G}f_g$, the following hold:
		\begin{enumerate}[(a)]
			\item $\|f_ga-af_g\|<\varepsilon$ for all $g\in G$ and all $a\in F$;
			\item $\|(\alpha_\omega)_g(f_h)-f_{gh}\|<\varepsilon$ for all $g,h\in G$;
			\item $(1-f-\varepsilon)_+\precsim_{A_\omega} x$;
			\item $\|f\|>1-\varepsilon$.
		\end{enumerate}
	\end{enumerate}
\end{proposition}
\begin{proof}
	We only prove the case of the weak tracial Rokhlin property.\\
	(1) $\Rightarrow$ (2): Let $S$ be a separable subset of $A_\omega$, and let $x\in (A_\omega)_+\setminus\{0\}$. We may assume that $\|x\|=1$. It is well known that positive elements lift to positive elements. Lift $x$ to a sequence $(x_n)_{n=1}^\infty$ in $l^\infty(A)_+$. Since $\lim\limits_{n\rightarrow \omega}\|x_n\|=1$, by replacing $x_n$ with $x_n/\|x_n\|$ we may further assume that $\|x_n\|=1$ for all $n\in \mathbb{N}$. Choose a dense sequence $(s_i)_{i=1}^\infty$ in $S$, and lift each $s_i$ to a bounded sequence $(s_i^{(n)})_{n=1}^\infty$ in $l^\infty(A)$. For each $n\in \mathbb{N}$, applying Theorem \ref{the1}(3) with $F_n=\{s_1^{(n)},s_2^{(n)},\cdots,s_n^{(n)}\}$ in place of $F$, with $1/n$ in place of $\varepsilon$, and with $(x_n-1/2)_+$ in place of $x$, we obtain orthogonal positive contractions $f_g^{(n)}\in A$ for $g\in G$ such that, with $f^{(n)}=\sum_{g\in G}f_g^{(n)}$, the following hold:
	\begin{enumerate}[(a$'$)]
		\item $\lim_{n\rightarrow \infty}\|f_g^{(n)}s_i^{(n)}-s_i^{(n)}f_g^{(n)}\|=0$ for all $g\in G$ and for all $i\in \mathbb{N}$;
		\item $\lim_{n\rightarrow \infty}\|\alpha_g(f_h^{(n)})-f_{gh}^{(n)}\|=0$ for all $g,h\in G$;
		\item $1-f^{(n)}\precsim (x_n-1/2)_+$ for all $n\in \mathbb{N}$;
		\item $\lim_{n\rightarrow \infty}(\|s_i^{(n)}\|-\|f^{(n)}s_i^{(n)}f^{(n)}\|)=0$ for all $i\in \mathbb{N}$.
	\end{enumerate}
	Set $f_g=\pi(f_g^{(1)},f_g^{(2)},\cdots)\in A_\omega$ for $g\in G$, where $\pi:l^\infty(A)\rightarrow A_\omega$ is the quotient map. Then $(f_g)_{g\in G}$ is a family of orthogonal positive contractions in $A_\omega$. Since $(s_i)_{i=1}^\infty$ is dense in $S$, we have $f_g\in A_\omega\cap S'$ for $g\in G$, and Condition (a) and (c) hold. By (c$'$) and Lemma \ref{lem1}(2), for each $n\in\mathbb{N}$ there is $v_n\in A$ such that $\|1-f^{(n)}-v_nx_nv_n^*\|<1/n$ and $\|v_n\|\leq\sqrt{2}\|1-f^{(n)}\|^{1/2}\leq \sqrt{2}$. Thus $(v_n)_{n=1}^\infty$ is a bounded sequence. Set $v=\pi(v_1,v_2,\cdots)\in A_\omega$, then $1-f=vxv^*\precsim_{A_\omega} x$. Condition (b) holds.\\
	(2) $\Rightarrow$ (3) is immediate.\\
	(3) $\Rightarrow$ (1): Let $F\subset A$ be a finite subset, let $0<\varepsilon<1/2$, and let $x\in A_+\setminus\{0\}$. We may assume that $F\subset A^1$. Applying the hypothesis with $F,\varepsilon,x$ as given, we obtain a family of orthogonal positive contractions $(e_g)_{g\in G}$ in $A_\omega$ satisfying Condition (a)\textendash(d) of (3). By \cite[Lemma 10.1.12]{MR1420863}, we can lift $(e_g)_{g\in G}$ to a family of orthogonal positive contractions $((e_g^{(n)})_{n=1}^\infty)_{g\in G}$ in $l^\infty(A)$. Since $(1-e-\varepsilon)_+\precsim_{A_\omega} x$, there is $v\in A_\omega$ such that $\|(1-e-\varepsilon)_+-vxv^*\|<\varepsilon$. Lift $v$ to a bounded sequence $(v_n)_{n=1}^\infty$ in $l^\infty(A)$ and set $e^{(n)}=\sum_{g\in G}e_g^{(n)}$. Then we can choose a sufficiently large $k$ such that
	\begin{enumerate}[(a$''$)]
		\item $\|e_g^{(k)}a-ae_g^{(k)}\|<\varepsilon$ for all $g\in G$ and all $a\in F$;
		\item $\|\alpha_g(e_h^{(k)})-e_{gh}^{(k)}\|<\varepsilon$ for all $g,h\in G$;
		\item $\|(1-e^{(k)}-\varepsilon)_+-v_kxv_k^*\|<\varepsilon$;
		\item $\|e^{(k)}\|>1-\varepsilon$.
	\end{enumerate}
	Define a continuous function $\varphi:[0,1]\rightarrow [0,1]$ by
	\begin{equation*}
		\varphi(t) = \begin{cases}
			\frac{1}{1-2\varepsilon}t \quad & 0\leq t\leq 1-2\varepsilon\\
			1 \quad & 1-2\varepsilon\leq t \leq 1.
		\end{cases}
	\end{equation*}
	Set $f_g=\varphi(e_g^{(k)})$ for each $g\in G$, and $f=\sum_{g\in G}f_g$. Then $(f_g)_{g\in G}$ is a family of orthogonal positive contractions in $A$. Since $\|f_g-e_g^{(k)}\|<2\varepsilon$ for all $g\in G$, it is easy to see that Condition (a) and (b) in Theorem \ref{the1}(ii) hold with $\varepsilon$ replaced by $5\varepsilon$. Since $\varphi(0)=0$ and $(e_g^{(k)})_{g\in G}$ are pairwise orthogonal, we have
	\[
	f=\sum_{g\in G}f_g=\sum_{g\in G}\varphi(e_g^{(k)})=\varphi(\sum_{g\in G}e_g^{(k)})=\varphi(e^{(k)}).
	\]
	Thus $\|f\|=1$ by (d)$''$. Condition (d) in Theorem \ref{the1}(ii) holds. Using (c)$''$ and Lemma \ref{lem1}(1) at the fourth step, we get
	\[
	1-f=1-\varphi(e^{(k)})=\frac{1}{1-2\varepsilon}(1-e^{(k)}-2\varepsilon)_+\sim (1-e^{(k)}-2\varepsilon)_+\precsim v_kxv_k^*\precsim x,
	\]
	which is Condition (c) in Theorem \ref{the1}(ii). By Theorem \ref{the1}, $\alpha$ has the weak tracial Rokhlin property.\\
	The proof for the tracial Rokhlin property is the same but easier. In the proof of (3) implies (1), use Lemma \ref{lem6} to lift orthogonal projections $(e_g)_{g\in G}$ in $A_\omega$ to orthogonal projections $((e_g^{(n)})_{n=1}^\infty)_{g\in G}$ in $l^\infty(A)$. After choosing a sufficiently large $k$, we do not need to use functional calculus for $e_g^{(k)}$. Since $e^{(k)}$ is a projection, by (c$''$) we have
	\[
	1-e^{(k)}\sim (1-e^{(k)}-2\varepsilon)_+\precsim v_kxv_k^*\precsim x.
	\]
\end{proof}

Next we give the proof of Theorem \ref{the2}.

\begin{proof}[Proof of Theorem \ref{the2}]
	By Proposition \ref{pro1}, it is sufficient to show that for any finite subset $F\subset A$, any $\varepsilon>0$, and any $x\in A_+\setminus\{0\}$, there exist a family of non-zero orthogonal projections $(p_g)_{g\in G}$ in $A_\omega$ such that, with $p=\sum_{g\in G}p_g$, the following hold:
	\begin{enumerate}[(1)]
		\item $\|p_ga-ap_g\|<\varepsilon$ for all $g\in G$ and all $a\in F$;
		\item $(\alpha_\omega)_g(p_h)=p_{gh}$ for all $g,h\in G$;
		\item $1-p\precsim_{A_\omega} x$.
	\end{enumerate}
	Without loss of generality, we may assume that $F\subset A^1$. By replacing $F$ by $\bigcup_{g\in G}\alpha_g(F)$, we may further assume that $\alpha_g(F)=F$. Set $n=\text{card}(G)$. We claim that there is a non-zero positive element $y\in A$ such that
	\[
	y\oplus \bigoplus_{g\in G}\alpha_g(y)\precsim x.
	\]
	In fact, since $A$ is simple and non type I, by \cite[Lemma 2.4]{1408.5546} there is a non-zero positive element $z\in A$ such that $z\otimes 1_{n+1}\precsim x$. Then by Lemma \ref{lem4} there is a non-zero positive element $y\in A$ such that $y\precsim \alpha_{g^{-1}}(z)$ for all $g\in G$. Hence $y\oplus \bigoplus_{g\in G}\alpha_g(y)\precsim z\otimes 1_{n+1}\precsim x$. This proves the claim.\\
	Set $\varepsilon_1=\text{min}\{\varepsilon/3,1/4n\}$, and define a continuous function $\varphi:[0,1]\rightarrow [0,1]$ by
	\begin{equation*}
		\varphi(t) = \begin{cases}
			\frac{1}{\varepsilon_1}t \quad & 0\leq t\leq \varepsilon_1\\
			1 \quad & \varepsilon_1\leq t \leq 1.
		\end{cases}
	\end{equation*}
	Applying Lemma \ref{lem3} with this function $\varphi$ and with $\varepsilon_1$ as given, we get $\delta>0$ such that whenever $E\subset A$ is a subalgebra, and $d\in A$ and $e\in E$ satisfy
	\[
	\|d\|\leq 1, 0\leq e\leq 1, \|[e,d]\|<\delta, \text{ and dist}(ed,E)<\delta,
	\]
	then
	\[
	\|[\varphi(e),d]\|<\varepsilon_1, \text{ and dist}(\varphi(e)d,E)<\varepsilon_1.
	\]
	Applying the hypothesis with $F$ as given, with $\delta$ in place of $\varepsilon$, and with $y$ in place of $x$. We can find a separable $C^*$-subalgebra $B\subset A$, a positive element $c\in A\cap B'$ with $\|c\|=1$, such that the following hold:
	\begin{enumerate}[(1$'$)]
		\item $\|ca-ac\|<\delta$ for all $a\in F$;
		\item $ca\in_\delta B$ for all $a\in F$;
		\item $1-c\precsim y$;
		\item $A_\omega\cap B'$ has real rank zero. 
	\end{enumerate}
	Since $\alpha$ has the weak tracial Rokhlin property, applying Proposition \ref{pro1}(2) with $F\cup B\cup \{c\}$ in place of $S$, and $y$ in place of $x$, there are pairwise orthogonal positive contractions $(f_g)_{g\in G}$ in $A_\omega\cap(F\cup B\cup \{c\})'$ such that, with $f=\sum_{g\in G}f_g$, the following hold:
	\begin{enumerate}[(1$''$)]
		\item $(\alpha_\omega)_g(f_h)=f_{gh}$ for all $g,h\in G$;
		\item $1-f\precsim y$;
		\item $\|faf\|=\|a\|$ for all $a\in F\cup B\cup \{c\}$. 
	\end{enumerate}
	Set $c_1=(c-\varepsilon_1)_+,c_2=\varphi(c)$, then we have $c_1,c_2\in A\cap B',c_2c_1=c_1$ and $\|c_1-c\|\leq \varepsilon_1$.
	Since $(f_g)_{g\in G}$ are pairwise orthogonal and all commute with $c_1$, we have $\|fc_1f\|=\text{max}\{\|f_gc_1f_g\|:g\in G\}$. Hence there exists $g_0\in G$ such that $\|f_{g_0}c_1f_{g_0}\|=\|fc_1f\|=\|c_1\| \geq 1-\varepsilon_1>0$. Since $A_\omega\cap B'$ has real rank zero, by \cite[Theorem 3.2.5]{MR1884366} there is a non-zero projection $p_1\in \overline{c_1f_{g_0}(A_\omega\cap B')f_{g_0}c_1}$ such that
	\[
	\|p_1c_1f_{g_0}p_1-c_1f_{g_0}\|<1/{4n}.
	\]
	Set $p_g=(\alpha_\omega)_g(p_1)$ for $g\in G\setminus\{1\}$, and $p=\sum_{g\in G}p_g$. Since $p_1\in \overline{c_1f_{g_0}(A_\omega\cap B')f_{g_0}c_1}\subset \overline{f_{g_0}A_\omega f_{g_0}}$, we have $p_g\in \overline{f_{gg_0}A_\omega f_{gg_0}}$. Thus $(p_g)_{g\in G}$ is a family of non-zero orthogonal projections in $A_\omega$. We will show that $(p_g)_{g\in G}$ satisfy (1)\textendash(3).\\
	Note that (2) follows from the definition of $p_g$. To see (1), let $a\in F$. By (1$'$) and the choice of $\delta$, we have $\|[c_2,a]\|<\varepsilon_1$. By (2$'$) and the choice of $\delta$, there is $b\in B$ such that $\|c_2a-b\|<\varepsilon_1$. Since $c_2c_1=c_1$ and $p_1\in \overline{c_1f_{g_0}(A_\omega\cap B')f_{g_0}c_1}$, we have $c_2p_1=p_1=p_1c_2$. Thus
	\begin{align*}
		\|p_1a-ap_1\| & =\|p_1c_2a-ac_2p_1\|\\
		& \leq \|p_1c_2a-p_1b\|+\|p_1b-bp_1\|+\|bp_1-c_2ap_1\|+\|c_2ap_1-ac_2p_1\|\\
		& \leq \|p_1\|\cdot\|c_2a-b\|+\|b-c_2a\|\cdot\|p_1\|+\|c_2a-ac_2\|\cdot\|p_1\|\\
		& < 3\varepsilon_1\leq \varepsilon.
	\end{align*}
	Since $a\in F$ is arbitrary, we get $\|p_1a-ap_1\|<\varepsilon$ for all $a\in F$. Using (2) and $\alpha_g(F)=F$ for all $g\in G$, we have
	\begin{align*}
		\|p_ga-ap_g\| & =\|(\alpha_\omega)_g[p_1\alpha_{g^{-1}}(a)-\alpha_{g^{-1}}(a)p_1]\|\\
		& =\|p_1\alpha_{g^{-1}}(a)-\alpha_{g^{-1}}(a)p_1\|<\varepsilon.
	\end{align*}
	This proves (1). To see (3), note that
	\[
	1-p=1-\sum_{g\in G}p_g\leq 1-\sum_{g\in G}p_g(\alpha_\omega)_g(c_1f_{g_0})p_g,
	\]
	and
	\begin{align*}
		& \|[1-\sum_{g\in G}p_g(\alpha_\omega)_g(c_1f_{g_0})p_g]-[1-f+\sum_{g\in G}f_{gg_0}^{1/2}(\alpha_\omega)_g(1-c)f_{gg_0}^{1/2}]\|\\
		={} & \|[1-\sum_{g\in G}(\alpha_\omega)_g(p_1c_1f_{g_0}p_1)]-[1-\sum_{g\in G}(\alpha_\omega)_g(cf_{g_0})]\|\\
		\leq{} & \sum_{g\in G}\|(\alpha_\omega)_g[p_1c_1f_{g_0}p_1-cf_{g_0}]\|\\
		\leq{} & n\|p_1c_1f_{g_0}p_1-c_1f_{g_0}\|+n\|c_1-c\|\cdot\|f_{g_0}\|\\
		<{} & 1/4+n\varepsilon_1\leq 1/2.
	\end{align*}
	Thus, using Lemma \ref{lem1}(3) at the second step, Lemma \ref{lem1}(1) at the third step, (3$'$) and (2$''$) in the fourth step, we get
	\begin{align*}
		1-p & \sim_{A_\omega} (1-p-1/2)_+\\
		& \precsim_{A_\omega} (1-\sum_{g\in G}p_g(\alpha_\omega)_g(c_1f_{g_0})p_g-1/2)_+\\
		& \precsim_{A_\omega} 1-f+\sum_{g\in G}f_{gg_0}^{1/2}(\alpha_\omega)_g(1-c)f_{gg_0}^{1/2}\\
		& \precsim_{A_\omega} y\oplus \bigoplus_{g\in G}(\alpha_\omega)_g(y)\\
		& \precsim_{A_\omega} x.
	\end{align*}
	This proves (3).
\end{proof}

\begin{remark}~
	\begin{enumerate}[(1)]
		\item Every simple unital $\rm{C}^*$-algebra with tracial rank zero satisfies the conditions in Theorem \ref{the2}. In fact, if $A$ is a simple unital $\rm{C}^*$-algebra with tracial rank zero, then there exist a finite dimensional subalgebra $B\subset A$ and a positive element $c=1_B$ satisfying Condition (1)\textendash(3) in Theorem \ref{the2}. Since $A$ has real rank zero(see \cite[Theorem 3.6.11]{MR1884366}), then $A_\omega$ also has real rank zero. Note that $A_\omega\cap B'$ is a direct sum of unital corners of $A_\omega$, so it has real rank zero. Condition (4) holds. 
		\item Recall that a Kirchberg algebra is a purely infinite, simple, nuclear, separable $\rm{C}^*$-algebra. Every unital Kirchberg algebra satisfies the conditions in Theorem \ref{the2}. In fact, if $A$ is a unital Kirchberg algebra, then $A_\omega\cap A'$ is simple and purely infinite by \cite[Theorem 2.12(3)]{MR2265050}, thus it has real rank zero by \cite{MR1010004}. We can take $B=A$ and $c=1_A$ to satisfy Condition (1)\textendash(4) in Theorem \ref{the2}.
	\end{enumerate}
\end{remark}

\begin{remark} Let $A$ be an infinite-dimensional simple unital $\rm{C}^*$-algebra, let $G$ be a finite group, and let $\alpha:G\rightarrow \text{Aut}(A)$ be an action of $G$ on $A$.
	\begin{enumerate}[(1)]
		\item If $A$ has tracial rank zero, then $\alpha$ has the tracial Rokhlin property if and only if $\alpha$ has the weak tracial Rokhlin property. This result was obtained by Phillips in \cite[Theorem 1.9]{MR3347172}.
		\item If $A$ is a Kirchberg algebra, then $\alpha$ has the tracial Rokhlin property if and only if $\alpha$ has the weak tracial Rokhlin property. This result was obtained by Gardella, Hirshberg and Santiago in \cite[Theorem 3.11]{MR4177290}. Indeed, they showed that $\alpha$ has the weak tracial Rokhlin property if and only if $\alpha$ is pointwise outer.
	\end{enumerate}
\end{remark}

\section*{Acknowledgements}
This work was supported by National Natural Science Foundation of China [Grant no. 11871375].




\bibliographystyle{plain}
\bibliography{ref_rp}

\begin{thebibliography}{10}

\bibitem{MR4078704}
Dawn~E. Archey and N.~Christopher Phillips.
\newblock Permanence of stable rank one for centrally large subalgebras and
  crossed products by minimal homeomorphisms.
\newblock {\em J. Operator Theory}, 83(2):353--389, 2020.

\bibitem{MR2712082}
Dawn~Elizabeth Archey.
\newblock {\em Crossed product {$C^*$}-algebras by finite group actions with a
  generalized tracial {R}okhlin property}.
\newblock ProQuest LLC, Ann Arbor, MI, 2008.
\newblock Thesis (Ph.D.)--University of Oregon.

\bibitem{MR4336489}
M.~Ali Asadi-Vasfi, Nasser Golestani, and N.~Christopher Phillips.
\newblock The {C}untz semigroup and the radius of comparison of the crossed
  product by a finite group.
\newblock {\em Ergodic Theory Dynam. Systems}, 41(12):3541--3592, 2021.

\bibitem{MR2188261}
B.~Blackadar.
\newblock {\em Operator algebras}, volume 122 of {\em Encyclopaedia of
  Mathematical Sciences}.
\newblock Springer-Verlag, Berlin, 2006.
\newblock Theory of $C^*$-algebras and von Neumann algebras, Operator Algebras
  and Non-commutative Geometry, III.

\bibitem{MR4216445}
Marzieh Forough and Nasser Golestani.
\newblock The weak tracial {R}okhlin property for finite group actions on
  simple {$\rm C^*$}-algebras.
\newblock {\em Doc. Math.}, 25:2507--2552, 2020.

\bibitem{MR4464578}
Xuanlong Fu and Huaxin Lin.
\newblock Tracial approximation in simple {$C^*$}-algebras.
\newblock {\em Canad. J. Math.}, 74(4):942--1004, 2022.

\bibitem{MR4177290}
Eusebio Gardella, Ilan Hirshberg, and Luis Santiago.
\newblock Rokhlin dimension: duality, tracial properties, and crossed products.
\newblock {\em Ergodic Theory Dynam. Systems}, 41(2):408--460, 2021.

\bibitem{MR4215379}
Guihua Gong, Huaxin Lin, and Zhuang Niu.
\newblock A classification of finite simple amenable {$\mathcal{Z}$}-stable
  {$C^\ast$}-algebras, {I}: {$C^\ast$}-algebras with generalized tracial rank
  one.
\newblock {\em C. R. Math. Acad. Sci. Soc. R. Can.}, 42(3):63--450, 2020.

\bibitem{MR647069}
Richard~H. Herman and Vaughan F.~R. Jones.
\newblock Period two automorphisms of {${\rm UHF}$} {$C^{\ast} $}-algebras.
\newblock {\em J. Functional Analysis}, 45(2):169--176, 1982.

\bibitem{MR702960}
Richard~H. Herman and Vaughan F.~R. Jones.
\newblock Models of finite group actions.
\newblock {\em Math. Scand.}, 52(2):312--320, 1983.

\bibitem{MR3063095}
Ilan Hirshberg and Joav Orovitz.
\newblock Tracially {$\mathcal{Z}$}-absorbing {$C^\ast$}-algebras.
\newblock {\em J. Funct. Anal.}, 265(5):765--785, 2013.

\bibitem{MR2053753}
Masaki Izumi.
\newblock Finite group actions on {$C^*$}-algebras with the {R}ohlin property.
  {I}.
\newblock {\em Duke Math. J.}, 122(2):233--280, 2004.

\bibitem{MR2047851}
Masaki Izumi.
\newblock Finite group actions on {$C^*$}-algebras with the {R}ohlin property.
  {II}.
\newblock {\em Adv. Math.}, 184(1):119--160, 2004.

\bibitem{MR2265050}
Eberhard Kirchberg.
\newblock Central sequences in {$C^*$}-algebras and strongly purely infinite
  algebras.
\newblock In {\em Operator {A}lgebras: {T}he {A}bel {S}ymposium 2004}, volume~1
  of {\em Abel Symp.}, pages 175--231. Springer, Berlin, 2006.

\bibitem{MR1759891}
Eberhard Kirchberg and Mikael R{\o}rdam.
\newblock Non-simple purely infinite {$C^\ast$}-algebras.
\newblock {\em Amer. J. Math.}, 122(3):637--666, 2000.

\bibitem{MR1906257}
Eberhard Kirchberg and Mikael R{\o}rdam.
\newblock Infinite non-simple {$C^*$}-algebras: absorbing the {C}untz algebras
  {$\mathscr O_\infty$}.
\newblock {\em Adv. Math.}, 167(2):195--264, 2002.

\bibitem{MR1124300}
Hua~Xin Lin and Shuang Zhang.
\newblock On infinite simple {$C^*$}-algebras.
\newblock {\em J. Funct. Anal.}, 100(1):221--231, 1991.

\bibitem{MR1884366}
Huaxin Lin.
\newblock {\em An introduction to the classification of amenable
  {$C^*$}-algebras}.
\newblock World Scientific Publishing Co., Inc., River Edge, NJ, 2001.

\bibitem{MR1420863}
Terry~A. Loring.
\newblock {\em Lifting solutions to perturbing problems in {$C^*$}-algebras},
  volume~8 of {\em Fields Institute Monographs}.
\newblock American Mathematical Society, Providence, RI, 1997.

\bibitem{MR2954514}
Hiroki Matui and Yasuhiko Sato.
\newblock {$\mathcal{Z}$}-stability of crossed products by strongly outer
  actions.
\newblock {\em Comm. Math. Phys.}, 314(1):193--228, 2012.

\bibitem{MR2808327}
N.~Christopher Phillips.
\newblock The tracial {R}okhlin property for actions of finite groups on
  {$C^\ast$}-algebras.
\newblock {\em Amer. J. Math.}, 133(3):581--636, 2011.

\bibitem{1408.5546}
N.~Christopher Phillips.
\newblock Large subalgebras, 2014.
\newblock preprint, arXiv:1408.5546v1.

\bibitem{MR3347172}
N.~Christopher Phillips.
\newblock Finite cyclic group actions with the tracial {R}okhlin property.
\newblock {\em Trans. Amer. Math. Soc.}, 367(8):5271--5300, 2015.

\bibitem{MR3153204}
Qingyun Wang.
\newblock {\em Tracial {R}okhlin {P}roperty and {N}on-{C}ommutative
  {D}imensions}.
\newblock ProQuest LLC, Ann Arbor, MI, 2013.
\newblock Thesis (Ph.D.)--Washington University in St. Louis.

\bibitem{MR2545617}
Wilhelm Winter and Joachim Zacharias.
\newblock Completely positive maps of order zero.
\newblock {\em M\"{u}nster J. Math.}, 2:311--324, 2009.

\bibitem{MR1010004}
Shuang Zhang.
\newblock A property of purely infinite simple {$C^*$}-algebras.
\newblock {\em Proc. Amer. Math. Soc.}, 109(3):717--720, 1990.

\end{thebibliography}

\end{document}